\documentclass[11pt]{amsart}
\usepackage[margin=1in]{geometry}
\usepackage{amsmath,amssymb,amsthm}
\usepackage{booktabs}
\usepackage{graphicx}
\usepackage[hidelinks]{hyperref}
\usepackage{microtype}

\newtheorem{theorem}{Theorem}
\newtheorem{proposition}[theorem]{Proposition}

\theoremstyle{remark}
\newtheorem{remark}[theorem]{Remark}

\title[How long is a game of Uno?]{How long is a game of Uno? A counting law for large hands and the Wild Draw Four hoard}
\author{Pragyaan Gaur}
\email{pragyaangaur12@gmail.com}
\date{8 October 2026}

\begin{document}
\begin{abstract}
We study the length of one round of Uno under the official rules, for $N$ players who each start with $M$ cards. The fastest possible round takes at least $M$ turns with two players and at least $2M-1$ turns with three or more, and both bounds are attained for small $M$. An exact search over random deals shows that real deals rarely allow them. Under the official rules a round has no maximum length. Under forced play the probability that a round lasts more than $t$ turns decays exponentially in $t$ at a rate that does not depend on $M$. Our main result concerns very large hands. If every card a player ever holds is eventually played, a round lasts $\frac{27}{19}NM$ turns to first order, and the constant depends only on the composition of the deck. The official rule that a Wild Draw Four may only be played by a player who holds no card of the current colour breaks this law. Every large hand hoards Wild Draw Fours, $M/23$ of them per player, and the round ends with a dumping endgame. For an even number of players this gives $\frac{53}{46}NM$ turns. For an odd number of players the dumping cannot alternate round the table, and the endgame settles into a seating pattern whose turn rate we compute numerically. Simulations with up to $32{,}000$ cards per hand support these limits. Finally we solve a small strategy game: with two players the Nash equilibrium over a family of 48 strategies is a single pure strategy that wins 62.8\% of rounds against uniformly random play.
\end{abstract}

\maketitle

\section{Introduction}

Uno is one of the most widely played card games, and its rules are simple enough to analyse. Demaine, Demaine, Uehara, Uno and Uno~\cite{demaine2014} studied the computational complexity of generalised Uno. They showed that deciding whether a single player can play out a hand is NP-complete for generalised decks and polynomial in some restricted cases, and that the uncooperative two-player game is in P. Their paper is about deciding what is possible from a given position. This paper asks a different question. How many turns does a round take, at least, at most and on average, as a function of the number of players $N$ and the starting hand size $M$?

The answers come in three parts. Section~\ref{sec:short} gives exact bounds for the shortest round and measures how often real deals reach them. Section~\ref{sec:long} shows that the longest round is unbounded and describes the tail of the length distribution under forced play. Section~\ref{sec:law} contains the main result. A conservation argument gives the length of a round with large hands exactly to first order, and the official Wild Draw Four rule changes the answer through a mechanism that we call the hoard. Section~\ref{sec:games} reports a strategy tournament and its equilibria.

All code, data and figures are public~\cite{repo}. The simulations use a C engine that plays one round under the official rules and conserves all 108 cards, an exact solver for the shortest round, and an endgame simulator that stores hands as counts so that hoards can reach millions of cards.

\section{Rules and model}\label{sec:rules}

The deck has 108 cards. Each of four colours has one 0, two each of 1 to 9, and two each of Skip, Reverse and Draw Two. There are four Wilds and four Wild Draw Fours. On a turn a player plays one card that matches the top card of the discard pile in colour or in rank, or plays a Wild. A player who cannot play draws one card and plays it at once if it fits. A Draw Two or a Wild Draw Four makes the next player draw two or four cards and lose their turn. A Reverse acts as a Skip when $N=2$. A Wild Draw Four is legal only when the player holds no card of the current colour. We enforce this rule directly and do not model challenges. There is no stacking of draw cards. The first card turned up acts on the first player. A round ends when one player has no cards, and its length $T$ is the number of turns taken.

The official rules also allow a player to draw instead of playing a card that fits. Our simulated players never do this. They play a uniformly random legal card whenever one exists, so the simulations describe \emph{forced random play}. The shortest-round solver does allow the voluntary draw.

For large hands we use an \emph{infinite deck}, in which every card dealt or drawn is independent and has the distribution of a card drawn from the 108-card deck. This removes the bound $NM \le 107$ and lets $M$ grow without limit.

\section{The shortest round}\label{sec:short}

\begin{proposition}\label{prop:short}
With $N=2$ every round lasts at least $M$ turns. With $N\ge 3$ every round lasts at least $2M-1$ turns. On the 108-card deck the first bound is attained for every $M\le 25$, and the second is attained for every $N\ge3$ when $M\le5$.
\end{proposition}

\begin{proof}
The winner holds at least $M$ cards and plays one card per turn, so the winner takes at least $M$ turns. With two players this is the bound. Skip, Reverse and Draw Two all return the turn to the player who played it when $N=2$. The 24 such cards can be ordered so that each matches the next, by playing the six of one colour and then moving to the next colour through a card of the same rank. A first player who holds $M-1\le24$ of them in such an order, and one more card that fits after them, wins in $M$ turns from a suitable first card.

With three or more players no card returns the turn to the player who played it. A Skip or a draw card passes it two seats on and a Reverse passes it to the other neighbour. So between two turns of the winner there is at least one turn by another player, which gives at least $M+(M-1)=2M-1$ turns. For the construction, let the winner and one neighbour each hold $M-1$ Reverses, which needs $2(M-1)\le 8$ and so $M\le5$. The winner plays a Reverse, the turn goes to the neighbour, the neighbour plays a Reverse and the turn comes back. After $M-1$ such exchanges the winner plays a last card that fits.
\end{proof}

\begin{remark}
Other relays reach the bound for larger $M$ when $N$ is small. With $N=3$ a Skip or Draw Two from the winner passes the turn to the third player, whose next card returns it, and with $N=4$ two players who alternate Skips and Draw Twos do the same. The exact search below finds the bound for $N=3$ and $M=7$ in 8\% of deals.
\end{remark}

To see what real deals allow, we solved random deals exactly. The solver is given every hand and the order of the stock, and all players cooperate to end the round as fast as possible, with any player allowed to win. It uses iterative deepening A$^*$ search with the admissible bound of Proposition~\ref{prop:short} applied to every player. An independent breadth-first search agreed with it on all 145 test deals. Table~\ref{tab:short} and Figure~\ref{fig:short} show the results. With seven cards and two players the bound of 7 turns was never reached in 4000 deals, because it needs six turn-repeating cards in one hand. Three players reach their bound in 8\% to 11\% of deals for every $M$ we tried, because two neighbours who return the turn to each other are easy to find. With five and six players the search exceeded its node limit on many deals, and we omit those cases.

\begin{table}[h]
\centering
\begin{tabular}{rrrrrr}
\toprule
$N$ & $M$ & deals & bound & mean fastest round & share at the bound \\
\midrule
2 & 3 & 4000 & 3 & 7.08 & 2.9\% \\
2 & 5 & 4000 & 5 & 8.95 & 0.65\% \\
2 & 7 & 4000 & 7 & 11.47 & 0 \\
3 & 3 & 4000 & 5 & 7.16 & 11.1\% \\
3 & 5 & 2000 & 9 & 10.94 & 7.8\% \\
3 & 7 & 2000 & 13 & 14.95 & 8.1\% \\
4 & 3 & 4000 & 5 & 8.03 & 1.0\% \\
4 & 5 & 1000 & 9 & 12.94 & 0.1\% \\
4 & 7 & 1000 & 13 & 17.59 & 0 \\
\bottomrule
\end{tabular}
\caption{The fewest turns in which a random deal can end, with full information and full cooperation.}
\label{tab:short}
\end{table}

\begin{figure}[h]
\centering
\includegraphics[width=0.7\textwidth]{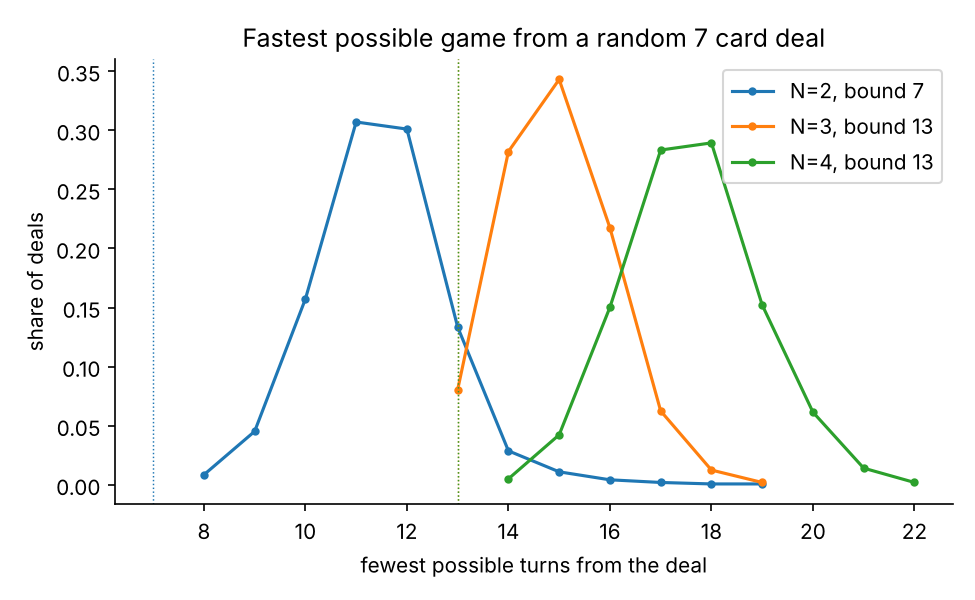}
\caption{Distribution of the fastest possible round over random seven-card deals. Dotted lines mark the bounds of Proposition~\ref{prop:short}.}
\label{fig:short}
\end{figure}

\section{The longest round and the typical round}\label{sec:long}

Under the official rules a round has no maximum length. Any player may draw instead of playing, and when the stock is empty the discard pile apart from its top card is shuffled to form a new stock. The players can therefore keep drawing for as long as they wish, and when every card is in a hand, a player who plays a card puts a new card back in circulation.

Forced random play is the interesting case. All 58 million rounds in our length experiments ended. With seven cards the mean length is 46.7 turns for two players, 50.1 for four and 83.1 for ten. Beyond the first few cards the mean grows linearly in $M$, by about 2.6 turns per card with two players and 5.4 with four, while the standard deviation barely changes with $M$. The length distribution has an exponential tail, $\Pr(T>t)\approx C e^{-\lambda t}$. With four players $\lambda\approx0.048$ per turn for every $M$ from 1 to 15 (Figure~\ref{fig:tail}). With seven cards $\lambda$ is 0.030 for two players and lies between 0.041 and 0.051 for three to ten. The tail does not depend on $M$ because a long round is a round stuck in the endgame, when every hand is small.

Seat order matters and the advantage does not fade as $M$ grows. With four players and seven cards the first player wins 26.0\% of rounds and the last player 24.3\%.

\begin{figure}[h]
\centering
\includegraphics[width=0.62\textwidth]{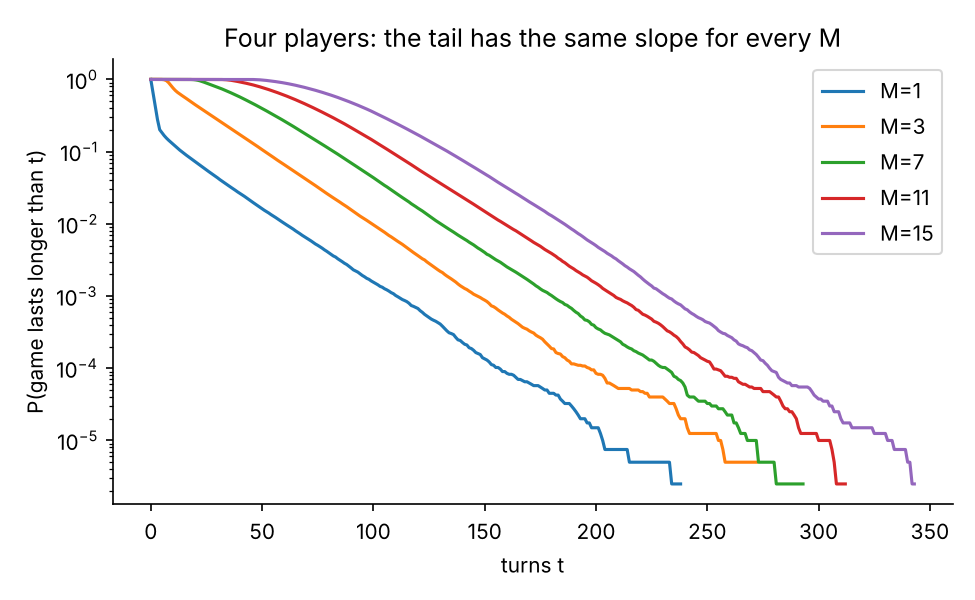}
\caption{Survival function of the round length with four players. The tail has the same slope for every starting hand size.}
\label{fig:tail}
\end{figure}

\section{Large hands}\label{sec:law}

\subsection{The counting law}

Let $p_{D2}=8/108$ and $p_{D4}=4/108$ be the fractions of Draw Twos and Wild Draw Fours in the deck.

\begin{theorem}\label{thm:law}
Consider a round on the infinite deck in which (i) every turn is a play, (ii) every card that any player holds during the round is played before the round ends, and (iii) the cards received by each player are independent of the play. Then the expected number of penalty cards drawn is $R=\frac{8}{19}NM$ and the expected number of turns is $\frac{27}{19}NM$.
\end{theorem}

\begin{proof}
The players together hold $NM+R$ cards over the round, where $R$ is the number of penalty cards. Every card held is dealt or drawn from the infinite deck, so in expectation $p_{D2}(NM+R)$ of them are Draw Twos and $p_{D4}(NM+R)$ are Wild Draw Fours. By (ii) all of them are played, and each played Draw Two and Wild Draw Four hands out two and four penalty cards. Hence
\[
R = (2p_{D2}+4p_{D4})(NM+R) = \tfrac{32}{108}(NM+R) = \tfrac{8}{27}(NM+R),
\]
so $R = \frac{8}{19}NM$. By (i) the number of turns equals the number of cards played, which is $NM+R=\frac{27}{19}NM$.
\end{proof}

Conditions (i) and (ii) hold to first order when $M$ is large and every hand is played down at the same rate. A large hand almost always holds a playable card, so almost every turn is a play. If the players are symmetric, their hands shrink together and every player is close to empty when the first player goes out. The theorem does not depend on the strategy, as long as the strategy plays every class of card as the hand runs down. We tested it with the Wild Draw Four made legal at any time, so that the hoard of the next section cannot form. Table~\ref{tab:free} shows $T/(NM)$. The values approach $27/19=1.42105$ slowly. Three effects explain the approach. The cards the losers still hold at the end are never played, which lowers $T$, and their share of $M$ falls from 0.042 at $M=500$ to 0.011 at $M=16000$ for two players. Every round has about 14 turns in which a player draws and cannot play, which raises $T$ by about $14/(NM)$. A smaller positive excess remains after these two corrections and shrinks as $M$ grows. Strategies that name their best colour, or play action cards first, also give $T/(NM)$ between 1.41 and 1.43 at $M=1000$. Strategies that hold back a class of cards do not satisfy (ii). A strategy that holds Wilds and saves action cards gives $T/(NM)=0.87$ at $M=1000$.

\begin{table}[h]
\centering
\begin{tabular}{rrrrrrr}
\toprule
$M$ & 500 & 1000 & 2000 & 4000 & 8000 & 16000 \\
\midrule
$N=2$ & 1.4403 & 1.4225 & 1.4144 & 1.4130 & 1.4138 & 1.4164 \\
$N=4$ & 1.3946 & 1.3947 & 1.3977 & 1.4024 & 1.4062 & 1.4110 \\
\bottomrule
\end{tabular}
\caption{$T/(NM)$ with the Wild Draw Four legal at any time, infinite deck, uniformly random play. Standard errors are below 0.0018. The limit predicted by Theorem~\ref{thm:law} is $27/19=1.42105$.}
\label{tab:free}
\end{table}

\subsection{The Wild Draw Four hoard}

Under the official rule a Wild Draw Four is legal only when the player holds no card of the current colour. A large hand always holds such a card, so it can never play a Wild Draw Four, and every Wild Draw Four it receives stays in the hand.

\begin{proposition}\label{prop:hoard}
Under the official rule, and under conditions (i) and (iii) of Theorem~\ref{thm:law} with (ii) applied to every card except Wild Draw Fours, each player plays $\frac{26}{23}M$ cards and receives $\frac{4}{23}M$ penalty cards in expectation before the hands run out of other cards, and then holds a hoard of $\frac{1}{23}M$ Wild Draw Fours.
\end{proposition}

\begin{proof}
Per player, the $M+R$ cards held contain $p_{D2}(M+R)$ Draw Twos, which are all played, and no Wild Draw Four is played. So $R = 2p_{D2}(M+R) = \frac{4}{27}(M+R)$, which gives $R=\frac{4}{23}M$ and $M+R=\frac{27}{23}M$. The hoard is $p_{D4}\cdot\frac{27}{23}M = \frac{1}{23}M$ and the cards played are $(1-p_{D4})\frac{27}{23}M = \frac{26}{23}M$.
\end{proof}

The same values come out of the fluid equations for a hand that plays a uniformly random legal card, with the top card treated as a fast Markov chain~\cite{repo}. In a traced two-player round with $M=1000$ the eventual winner reached 41 cards, nearly all Wild Draw Fours, against the predicted 43.

\subsection{The endgame}

Once a player holds only Wild Draw Fours, every one of them is legal, because the player holds no card of any colour. A Wild Draw Four makes the next player draw four cards and miss a turn. Let $h=M/23$ be the size of the hoard and let the endgame last $E h$ turns to first order. Then
\begin{equation}\label{eq:limit}
\frac{T}{NM} \to \frac{26}{23} + \frac{E}{23N}.
\end{equation}

\begin{proposition}\label{prop:even}
In the endgame in which every player starts with $h$ Wild Draw Fours and nothing else, if $N$ is even then $E=N/2$, so $T/(NM)\to\frac{53}{46}$ for every even $N$.
\end{proposition}

\begin{proof}
Seat 0 has no card of the current colour, so it plays a Wild Draw Four. Seat 1 draws four cards and is skipped, and seat 2, which also holds only Wild Draw Fours, does the same to seat 3. When $N$ is even this pattern closes round the table. The even seats dump and the odd seats only collect cards and never take a turn. Each round of the table takes $N/2$ turns and costs each even seat one card, so the first even seat goes out after $(N/2)h$ turns to first order. With two players the dumper keeps the turn because a Wild Draw Four acts as a Skip. Equation~\eqref{eq:limit} gives $\frac{26}{23}+\frac{1}{46}=\frac{53}{46}$.
\end{proof}

The endgame simulator gives $E=1,2,3,4$ for $N=2,4,6,8$, and a collecting seat ends with its own hoard and $4h$ new cards. So a loser holds $\frac{5N}{46(N-1)}M$ cards on average when the round ends, which is $\frac{5}{23}M$ for two players. This is the signature of the hoard. With the Wild Draw Four legal at any time, the losers' share of $M$ goes to zero.

When $N$ is odd the alternation cannot close. In simulated endgames the seats settle into a pattern with $(N-1)/2$ dumping seats at alternating positions, each followed by one collecting seat, and one place where two collecting seats sit next to each other. The second seat of that pair plays an ordinary card every time the turn reaches it, and its action cards disturb the pattern. The dumpers do not spend their hoards at the same rate. The dumper furthest from the adjacent pair is fastest, and $E$ is its number of turns per Wild Draw Four. We computed $E$ in two ways (Table~\ref{tab:odd}). The first is the full endgame with $h$ up to $10^6$. The second is a reduced model that keeps only the seating pattern. Its dumpers have endless hoards and small hands, and its collecting seats have hands so large that they always hold the current colour. When the collecting hands are given the deck's proportions, the reduced model is 1.0\% to 1.6\% low. When they are real hands that receive their penalty cards and play from what they hold, their composition drifts, and the reduced model matches the full endgame within its own sampling error. So the pattern explains the odd endgame completely. We have not found closed forms for these constants, because they depend on the drifting composition of the collecting hands.

\begin{table}[h]
\centering
\begin{tabular}{rcccc}
\toprule
$N$ & reduced, deck-like hands & reduced, real hands & full endgame & limit of $T/(NM)$ \\
\midrule
3 & 2.928 & 2.973 & $2.9717\pm0.0003$ & 1.1735 \\
5 & 3.989 & 4.055 & $4.0550\pm0.0003$ & 1.1657 \\
7 & 4.769 & 4.817 & $4.8164\pm0.0004$ & 1.1604 \\
9 & 5.702 & 5.744 & $5.7399\pm0.0003$ & 1.1582 \\
\bottomrule
\end{tabular}
\caption{The odd-$N$ endgame constant $E$. Reduced-model values are means of four runs of 600{,}000 turns and spread by about 0.005. The last column applies equation~\eqref{eq:limit}.}
\label{tab:odd}
\end{table}

\subsection{Full rounds}

Table~\ref{tab:official} shows $T/(NM)$ in full rounds under the official rule with the infinite deck. For even $N$ the data agree with $\frac{53}{46}=1.15217$. A fit of $L+aM^{-1/2}$ gives $L=1.15192\pm0.00034$ for two players and $L=1.15201\pm0.00029$ for four, and fits with a free exponent give exponents of $0.46\pm0.07$ and $0.45\pm0.08$. For odd $N$ the full rounds are still 0.004 to 0.006 below the predicted limits at $M=32000$. The endgame measured inside full rounds, counted from the first turn at which some player holds only Wild Draw Fours, lasts $2.95h$ turns for three players and $4.11h$ for five, close to Table~\ref{tab:odd}. So the remaining gap lies mostly in the main phase, where the even case shows the same slow approach.

\begin{table}[h]
\centering
\begin{tabular}{rrrrr}
\toprule
$M$ & $N=2$ & $N=3$ & $N=4$ & $N=5$ \\
\midrule
1000 & 1.1306 & 1.1805 & 1.1350 & 1.1438 \\
2000 & 1.1362 & 1.1728 & 1.1390 & 1.1489 \\
4000 & 1.1410 & 1.1687 & 1.1433 & 1.1523 \\
8000 & 1.1441 & 1.1684 & 1.1455 & 1.1555 \\
16000 & 1.1465 & 1.1692 & 1.1477 & 1.1582 \\
32000 & 1.1482 & 1.1699 & 1.1489 & 1.1603 \\
\midrule
limit & 1.1522 & 1.1735 & 1.1522 & 1.1657 \\
\bottomrule
\end{tabular}
\caption{$T/(NM)$ under the official rule, infinite deck, uniformly random play. Each entry averages 100 to 800 rounds, with standard errors from 0.0002 to 0.0016.}
\label{tab:official}
\end{table}

\begin{figure}[h]
\centering
\includegraphics[width=0.75\textwidth]{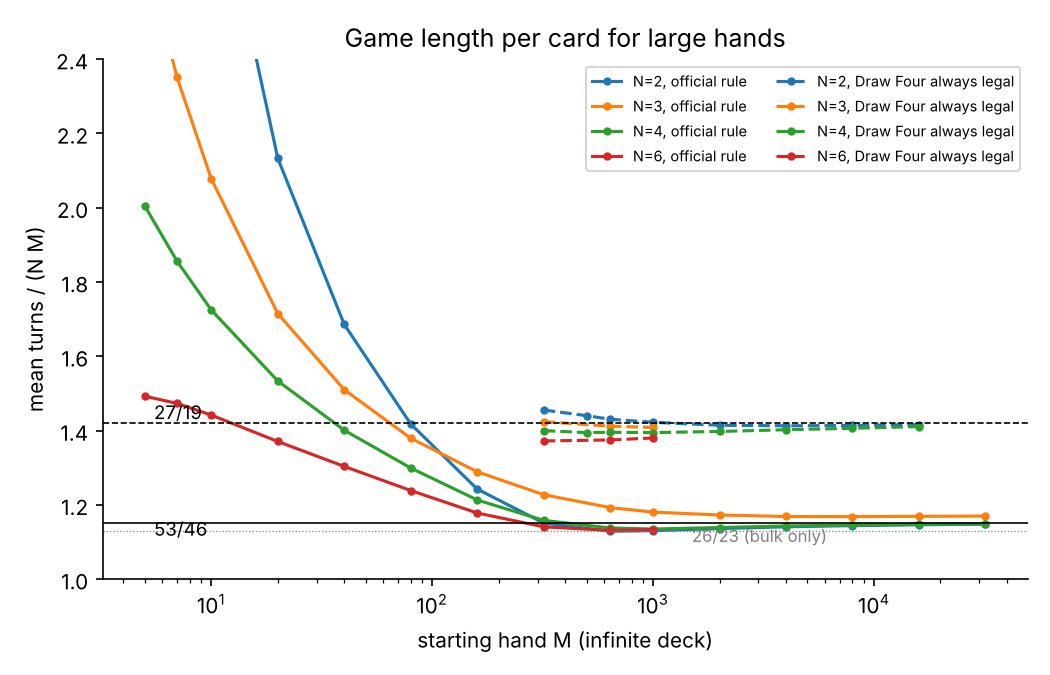}
\caption{$T/(NM)$ against $M$ for the official rule and for the Wild Draw Four legal at any time.}
\label{fig:fluid}
\end{figure}

\section{Strategies}\label{sec:games}

We defined a family of 48 strategies by five settings. A strategy may hold Wilds back while a coloured card fits, may play action cards first or last or neither, may name a random colour or the colour it holds most, may attack the next player with draw cards and Skips when that player holds two cards or fewer, and may prefer cards that keep play in its main colour. Uniformly random play is one member of the family.

With two players we played every ordered pair 80{,}000 times with seats swapped half the time, which gives a symmetric zero-sum matrix game. Its equilibrium, found by linear programming, is a single pure strategy: hold Wilds, play action cards first, name the colour held most and attack. It wins 62.8\% of rounds against random play. With two players every action card returns the turn, so playing them early is natural. Of the five settings, the colour choice matters most. On average over the other settings it raises the win rate against random play from 51\% to 58\%.

With four players we played each strategy as a single deviant against a field of three copies of each strategy, 20{,}000 rounds per pair. One field is stable in the sense that the best deviant wins 25.04\% against it. It holds Wilds, plays action cards like other cards, names its main colour, attacks and steers play towards its main colour. The two-player equilibrium wins only 23.6\% as a deviant in this field. Against three random players the stable strategy wins 32.8\%.

\section{Discussion}

The counting law shows that a large-hand Uno round behaves like a conservation process, in which the deck's share of draw cards fixes the number of extra cards that enter play. The Wild Draw Four rule is the one place where the official rules force every player to hold a class of cards back, and the hoard it creates sets the shape of the endgame. The even and odd cases differ because the dumping pattern either closes round the table or does not.

Several questions remain. The finite-size corrections, which look like $M^{-1/2}$, are fitted and not derived. The odd-$N$ endgame constants are computed and have no closed form. On the real 108-card deck $M$ cannot exceed $107/N$, so the large-hand results describe an idealised deck. The simulated players never draw voluntarily, and the strategy family is small and hand-built.

\section*{Code and data}
The C engine, the exact solver, the endgame simulator, the experiment scripts, their raw output and the scripts that draw the figures are available at \url{https://github.com/pragyaangaur/Uno-Reverse-Engineered}, version 1.0.0, under the MIT licence. Every number in this paper can be regenerated from that repository on a laptop.

\section*{Acknowledgements}
A large language model helped with the wording and the organization of this paper and with writing the simulation code. The author directed the work and checked all mathematical results, proofs, computations and code, and the author is responsible for all the content of this paper.

\end{document}